\documentclass[12pt,reqno]{amsart}
\usepackage{amsmath}
\usepackage{amssymb}
\usepackage{amstext}
\usepackage{mathrsfs}
\usepackage{a4wide}
\usepackage{graphicx}
\usepackage{bbm}
\usepackage{verbatim}
\usepackage{bm}
\usepackage[colorlinks, linkcolor=black, anchorcolor=black,
citecolor=black]{hyperref}
\usepackage{hyperref}
\hypersetup{hypertex=true,
            colorlinks=true,
            linkcolor=blue,
            anchorcolor=blue,
            citecolor=blue }

\allowdisplaybreaks \numberwithin{equation}{section}
\usepackage{color}
\usepackage{cases}

\numberwithin{equation}{section}

\newtheorem{theorem}{Theorem}[section]
\newtheorem{proposition}[theorem]{Proposition}

\theoremstyle{definition}

\theoremstyle{remark}

\newtheorem{example}[theorem]{Example}

\newcommand{\mv}{\mathbf{v}}

\begin{document}
\title[A rigidity result for the Euler equations in an annulus]{A rigidity result for the Euler equations in an annulus}
	
	\author{Yuchen Wang, Weicheng Zhan}
	
\address{School of Mathematical Science
Tianjin Normal University, Tianjin, 300074,
P.R. China}
\email{wangyuchen@mail.nankai.edu.cn}

\address{School of Mathematical Sciences, Xiamen University, Xiamen, Fujian, 361005, P.R. China}
\email{zhanweicheng@amss.ac.cn}


	\begin{abstract}
We are concerned with rigidity properties of steady Euler flows in two-dimensional bounded annuli. We prove that in an annulus, a steady flow with no interior stagnation point and tangential boundary conditions is a circular flow, which addresses an open question proposed by F. Hamel and N. Nadirashvili in [J. Eur. Math. Soc., 25 (2023), no. 1, 323-368]. The proof is based on the study of the geometric properties of the streamlines of
the flow and on `local' symmetry properties for the non-negative solutions of semi-linear elliptic equations with a continuous nonlinearity.
	\end{abstract}
	
	\maketitle{\small{\bf Keywords:} The Euler equation, Circular flows, Rigidity property, Semi-linear elliptic equation. \\

	
	\section{Introduction and Main result}
Let $D \subset \mathbb{R}^2$ be a domain with a smooth boundary. Consider stationary solutions of the Euler equations for an ideal fluid:
	\begin{align}\label{1-1}
		\begin{cases}
			\mathbf{v}\cdot \nabla \mathbf{v} =-\nabla P&\text{in}\ \, D,\\
			\nabla\cdot\mathbf{v}=0\,\ \, \ \ \ \ \ \ \  \ \, &\text{in}\ \, D,
\end{cases}
	\end{align}
where $\mathbf{v}=(v_1, v_2)$ is the velocity field and $P$ is the scalar pressure. Here the solutions $\mathbf{v}$ and $P$ are always understood in the classical sense, that is, they are (at least) of class $C^1$ in $D$, and therefore satisfy \eqref{1-1} everywhere in $D$. We also assume the tangential boundary condition, namely, $\mathbf{v}$ is (at least) continuous up to the boundary and tangential there:
\begin{equation}\label{1-2}
  \mathbf{v}\cdot\mathbf{n}=0\ \ \ \text{on}\ \, \partial D,
\end{equation}
where $\mathbf{n}$ is the outward unit normal on $\partial D$.

We are interested in rigidity properties of steady solutions of the Euler equations. Specifically, we want to address the following fundamental question pertaining to steady configurations
of fluid motion:

\smallskip
\emph{Given a domain $D$ with symmetry, to what extent must steady fluid states $\mathbf{v}$ inherit the geometric symmetry properties of the domain\,?}
\smallskip

In recent years, significant progress has been made in this regard. Throughout the paper, the stagnation points of a flow $\mathbf{v}$ are the points $x$ for which $|\mv(x)|=0$. In \cite{HN3, HN2}, Hamel and Nadirashvili proved that in strips, half planes or the plane, a bounded steady flow with no stagnation point is a parallel shear flow. The same conclusion holds for a steady flow in a periodic channel; see \cite{Cons}. We refer the reader to \cite{ Chae, Cons, Coti, Gom, HN3, HN2, HN1, HN, Nad, Ruiz} for some other relevant results in this aspect.

We say that a flow $\mathbf{v}$ is a circular flow if $\mathbf{v}(x)$ is parallel to the vector $\mathbf{e}_\theta(x)=\left(-x_2/|x|, x_1/|x|\right)$ at every point $x\in D\backslash\{0\}$. Very recently, Hamel and Nadirashvili \cite{HN3} showed that in a bounded annulus, a steady flow with no stagnation point is a circular flow, namely the streamlines are concentric circles. For $0<a<b<+\infty$, let $\Omega_{a, b}=\{x\in \mathbb{R}^2: a<|x|<b\}$, $C_a=\{x\in \mathbb{R}^2: |x|=a\}$ and $C_b=\{x\in \mathbb{R}^2: |x|=b\}$.
\begin{theorem}[\cite{HN3}, Theorem 1.2]
  Let $\mv$ be a $C^2(\overline{\Omega_{a, b}})$ flow solving \eqref{1-1}-\eqref{1-2} with $D=\Omega_{a, b}$ and such that
  \begin{equation}\label{1-3}
    \left\{x\in \overline{\Omega_{a, b}}: |\mv(x)|=0 \right\}\subsetneq C_a\ \ \ \text{or}\ \ \     \left\{x\in \overline{\Omega_{a, b}}: |\mv(x)|=0 \right\}\subsetneq C_b.
  \end{equation}
  Then $|\mv|>0$ in $\overline{\Omega_{a, b}}$ and $\mv$ is a circular flow, and there is a $C^2([a, b])$ function $V$ with constant strict sign such that
  \begin{equation*}
    \mv(x)=V(|x|)\mathbf{e}_\theta(x)\ \ \ \text{for all }\, x\in \overline{\Omega_{a, b}}.
  \end{equation*}
\end{theorem}

Note that if $\mv$ has no stagnation point in $\overline{\Omega_{a, b}}$, then the condition \eqref{1-3} is fulfilled. In other words, a steady flow having no stagnation point in the \emph{closed} annulus $\overline{\Omega_{a, b}}$ must be a circular flow. Nevertheless, the condition \eqref{1-3} does not seem natural and should be due to technical reasons. The following example escapes the scope of this theorem (see footnote on page 350 of \cite{HN}).

\begin{example}\label{ex1}
  The smooth flow given by
  \begin{equation*}
    \mv(x)=(|x|-a)\mathbf{e}_\theta(x),\ \ \ P(x)=|x|^2/2-2a|x|+a^2\ln|x|
  \end{equation*}
clearly solves \eqref{1-1}-\eqref{1-2} with $D=\Omega_{a, b}$. Notice that $|\mv|=0$ on $C_a$, so it does not satisfy the condition \eqref{1-3}.
\end{example}

With this in mind, Hamel and Nadirashvili proposed the following open question on page 349 of \cite{HN}:

\begin{center}
  \emph{if $|\mv|>0$ in $\Omega_{a, b}$, then $\mv$ is a circular flow\,?}
\end{center}

The aim of this paper is to address this open question. Specifically, we have
\begin{theorem}\label{th}
   Let $\mv$ be a $C^2(\overline{\Omega_{a, b}})$ flow solving \eqref{1-1}-\eqref{1-2} with $D=\Omega_{a, b}$. If $|\mv|>0$ in $\Omega_{a, b}$, then $\mv$ is a circular flow.
\end{theorem}

The proof of Theorem \ref{th} will be provided in the next section.

\section{Proof of Theorem \ref{th}}
In this section, we give a proof to Theorem \ref{th}. We will follow the strategy in \cite{HN}. The proof is based on the study of the geometric properties of the streamlines of
the flow and on `local' symmetry properties for the non-negative solutions of semi-linear elliptic equations with a continuous nonlinearity. The proof consists of two main steps. The first step is to derive that the corresponding stream function satisfies an elliptic semi-linear equation, which can be done by studying the geometric properties of the streamlines of the flow. The second step is to establish symmetry properties for the non-negative solutions of the semi-linear elliptic equation, which leads to the desired conclusion.

\smallskip
\noindent \emph{Proof of Theorem \ref{th}:}
The flow $\mathbf{v}$ has a unique (up to additive constants) stream function $u:\overline{\Omega_{a, b}}\to \mathbb{R}$ of class $C^3(\overline{\Omega_{a, b}})$ such that
\begin{equation*}
  \nabla^\perp u=\mathbf{v},\ \ \ \text{that is}\ \ \ \partial_1 u=v_2\ \ \text{and}\ \ \partial_2u=-v_1,
\end{equation*}
in $\overline{\Omega_{a, b}}$, since $\mathbf{v}$ is divergence free and satisfies the tangential boundary condition \eqref{1-2}. The tangency condition $\mathbf{v}\cdot \mathbf{n}=0$ on $\partial \Omega_{a, b}$ also means that $u$ is constant on $C_a$ and constant on $C_b$. Up to normalization, we may assume, without loss of generality, that
\begin{equation*}
  u=1\  \text{on}\ \, C_a\ \ \ \text{and}\ \ \ u=0\ \, \text{on}\  C_b.
\end{equation*}
Since $\mathbf{v}$ has no stagnation point in $\Omega_{a, b}$, the stream function $u$ has no critical point in $\Omega_{a, b}$. It follows that
\begin{equation*}
  0<u(x)<1\ \ \ \text{for all}\ \ x\in \Omega_{a, b}.
\end{equation*}
By Proposition \ref{p-1} below, we see that there is a continuous function $f: [0, 1] \to \mathbb{R}$ such that
  \begin{equation*}
   \Delta u+f(u)=0\ \ \ \text{in}\ \ \Omega_{a, b}.
  \end{equation*}
It remains to show that $u$ is a radially decreasing function with respect to the origin. Set
\begin{equation*}
  \bar{u}(x)=\begin{cases}
               1, & \mbox{if }\ |x|\le a, \\
               u(x), & \mbox{if}\ a<|x|<b.
             \end{cases}
\end{equation*}
Then $\bar{u}\in  H^1_0(B_b)\cap C(\overline{B_b})$, where $B_b=\{x\in \mathbb{R}^2: |x|<b\}$. Moreover, $\bar{u}$ is a weak solution of the following problem
	\begin{align*}
		\begin{cases}
			-\Delta w=f(w),\ \ 0< w < 1&\text{in}\ \ \Omega_{a, b},\\
             w(x)\equiv 1,\ \ &\text{if}\ \ |x|\le a,\\
             w\in H^1_0(B_b).
\end{cases}
	\end{align*}
It follows from Proposition \ref{p-2} below that $\bar u$ is locally symmetric, namely, it is  radially symmetric and radially decreasing in some annuli (probably infinitely many) and flat elsewhere. Recall that $u$ has no critical point in $\Omega_{a, b}$. We conclude that the number of annuli can only be one at most, and hence $u$ is a radially decreasing function. The proof is thus complete.

Below we list two results used in the proof. The first result states that the corresponding stream function of a steady flow satisfies a semi-linear elliptic equation under the assumptions in Theorem \ref{th}.
\begin{proposition}\label{p-1}
  Let $\mathbf{v}$ be as in Theorem \ref{th} and let $u\in C^3(\overline{\Omega_{a, b}})$ be the corresponding stream function, and let $J$ its range defined by
  \begin{equation*}
    J=\{u(x)\mid x\in \overline{\Omega_{a, b}}\}.
  \end{equation*}
  Then there is a continuous function $f: J \to \mathbb{R}$ such that
  \begin{equation*}
   \Delta u+f(u)=0\ \ \ \text{in}\ \ \overline{\Omega_{a, b}}.
  \end{equation*}
\end{proposition}

\begin{proof}
This result has been pointed out in \cite{HN}. For the sake of completeness, here we give a detailed proof. Based on the discussion above, without loss of generality, we may assume that
\begin{equation}\label{4-1}
  u=1\  \text{on}\ \, C_a,\ \ \ u=0\ \, \text{on}\  C_b\ \ \ \text{and}\ \ \ 0<u<1\ \, \text{in}\ \Omega_{a, b}.
\end{equation}
Consider any point $y\in \Omega_{a, b}$. Let $\sigma_y$ be the solution of
	\begin{align}\label{3-10}
		\begin{cases}
			 \dot{\sigma}_y(t) = \nabla u(\sigma_y(t)), &\\
             \sigma_y(0)=y.\ \ &
\end{cases}
	\end{align}
Then by Lemma 2.2 in \cite{HN}, there are some quantities $t_y^{\pm}$ such that $-\infty\le t^-_y<0<t^+_y\le +\infty$ and the solution $\sigma_y$ of \eqref{3-10} is of class $C^1((t_y^-, t^+_y))$ and ranges in $\Omega_{a, b}$, with
	\begin{align}\label{3-11}
		\begin{cases}
		|\sigma_y(t)|\to a\ \,\text{and}\ \,u(\sigma_y(t))\to 1\ \,\text{as}\ \,t\to t_y^-,\ \ &\\
|\sigma_y(t)|\to b \ \,\text{and}\ \,u(\sigma_y(t))\to 0\ \,\text{as}\ \, t\to t_y^+. &
\end{cases}
	\end{align}
Set $g:=u\circ \sigma_y\in C^1((t_y^-, t_y^+))$. Then $g$ is increasing since $(u\circ \sigma_y)'(t)=|\nabla u(\sigma_y(t))|^2=|\mathbf{v}(\sigma_y(t))|^2>0$ for all $t\in (t_y^-, t_y^+)$, and hence $g$ is an increasing homeomorphism from $(t_y^-, t_y^+)$ onto $(0, 1)$. Consider the function $f: (0, 1)\to \mathbb{R}$ defined by
\begin{equation}\label{3-12}
  f(\tau)=-\Delta u(\sigma_y(g^{-1}(\tau)))\ \ \ \text{for}\ \ \tau\in (0, 1).
\end{equation}
Then $f$ is of class $C^1((0, 1))$ by the chain rule. The equation $\Delta u+f(u)=0$ is now satisfied along the curve $\sigma_y((t_y^-, t_y^+))$. Let us check it in the whole set $\Omega_{a, b}$. Consider first any point $x\in \Omega_{a, b}$. Let $\xi_x$ be the solution of
	\begin{align*}
		\begin{cases}
			 \dot{\xi}_x(t) = \mathbf{v}(\xi_x(t)), &\\
             \xi_x(0)=x.\ \ &
\end{cases}
	\end{align*}
Then $\xi_x$ is defined in $\mathbb{R}$ and periodic. Furthermore, the streamline $\Phi_x:=\xi_x(\mathbb{R})$ is a $C^1$ Jordan curve surrounding the origin in $\Omega_{a, b}$ and meets the curve $\sigma_y((t_y^-, t_y^+))$ once; see Lemma 2.6 in \cite{HN}. Hence, there is $s\in (t_y^-, t_y^+)$ such that $\sigma_y(s)\in \Phi_x$. Note that both the stream function $u$ and the vorticity $\Delta u$ are constant along the streamline $\Phi_x$. It follows from \eqref{3-12} that
\begin{equation*}
  \Delta u(x)+f(u(x))=\Delta u(\sigma_y(s))+f(u(\sigma_y(s)))=\Delta u(\sigma_y(s))+f(g(s))=0.
\end{equation*}
Therefore, $\Delta u+f(u)=0$ in $\Omega_{a, b}$. By the continuity of $u$ in $\overline{\Omega_{a, b}}$ and \eqref{4-1}, we have
\begin{equation*}
 \max_{t\in \mathbb{R}}|\xi_x(t)|\to a  \ \text{as}\ |x|\to a\ \ \ \text{and}\ \ \   \min_{t\in \mathbb{R}}|\xi_x(t)|\to b\ \text{as}\ |x|\to b.
\end{equation*}
Since $\Delta u$ is uniformaly continuous in $\overline{\Omega_{a, b}}$ and constant along any streamline of the flow, we see that $\Delta u$ is constant on $C_a$ and constant on $C_b$. Call $d_1$ and $d_2$ the values of $\Delta u$ on $C_a$ and $C_b$, respectively. Set $f(1)=-d_1$ and $f(0)=-d_2$. Then we infers from \eqref{3-11} and \eqref{3-12} that $f:[0, 1]\to \mathbb{R}$ is continuous in $[0, 1]$ and that the equation $\Delta u+f(u)=0$ holds in $\overline{\Omega_{a, b}}$. The proof is thereby complete.
\end{proof}

\begin{example}
  Let $\mv$ be the steady flow in Example \ref{ex1}, i.e., $ \mv(x)=(|x|-a)\mathbf{e}_\theta(x)$. Then
  \begin{equation*}
\begin{split}
   u(x) & =(|x|-a)^2/2\ \ \ \text{(up to additive constants)}, \\
    J & = \{u(x)\mid x\in \overline{\Omega_{a, b}}\}=[0, (b-a)^2/2].
\end{split}
  \end{equation*}
  Let $f(s)  =-2+a/(a+\sqrt{2s}),\ s\in J$. One can easily verify that the equation $\Delta u+f(u)=0$ holds in $\overline{\Omega_{a, b}}$.
\end{example}

The following result concerns the symmetry of solutions to semi-linear elliptic equations with a continuous nonlinearity, which follows from Theorem 6.1 and Corollary 7.6 in \cite{Bro1}. Let $B_R(z)$ (resp. $Q_R(z)$) be the open (resp. closed) ball in $\mathbb{R}^N$, with radius $R>0$ centered at $z\in \mathbb{R}^N$.
\begin{proposition}[\cite{Bro1}]\label{p-2}
  Let $0<r<R<+\infty$ and let $g=g(r,t)$ be of class $ C([0, R]\times [0, +\infty))$ and be non-increasing in $r$. Let $u\in H^1_0(B_R(0))$ be a weak solution of the following problem
	\begin{align*}
		\begin{cases}
			-\Delta u=f(|x|, u),\ \ 0\le u \le 1&\text{in}\ \ B_R(0)\backslash Q_r(0),\\
             u\equiv 1,\ \ &\text{in}\  \  Q_r(0).
\end{cases}
	\end{align*}
In addition, suppose that the set $U:=\{x\in B_R(0): 0<u(x)<1\}$ is open and $u\in C^1(U)$. The $u$ is locally symmetric in the following sense:
\begin{equation*}
  U=A\cup S,\ \ A=\bigcup_{k\in K} A_k,
\end{equation*}
where
  \begin{itemize}
    \item [(1)]$K$ is a countable set,

    \smallskip
    \item [(2)]$A_k$ are disjoint open annuli $B_{R_k}(z_k)\backslash Q_{r_k}(z_k)$ with $R_k>r_k\ge 0$, $z_k\in \mathbb{R}^N$,

    \smallskip
    \item [(3)]$u$ is radially symmetric and decreasing in each domain $A_k$, more precisely,
    \begin{equation*}
      u=u(|x-z_k|),\ \ \ \frac{\partial u}{\partial \rho}<0\ \text{in}\ A_k\ \ \ (\rho=|x-z_k|)
    \end{equation*}
    and
    \begin{equation*}
      u(x)\ge u \big{|}_{\partial Q_{r_k}(z_k)},\ \ \ \forall\, x\in B_{r_k}(z_k),
    \end{equation*}

    \smallskip
    \item [(4)]$\nabla u=0$ on the set $S$.
  \end{itemize}
\end{proposition}

Note that $\partial A_k\cap U\subset S$, and hence $\nabla u(x)=0$ for any $x\in \partial A_k\cap U$.


	\phantom{s}
	\thispagestyle{empty}

\end{document}